\documentclass[11pt,oneside,english]{amsart}
\usepackage[authoryear]{natbib}

\RequirePackage[colorlinks,citecolor=blue,urlcolor=blue]{hyperref}

\usepackage[T1]{fontenc}
\usepackage[latin9]{inputenc}
\usepackage{amsthm}
\usepackage{amstext}
\usepackage{amssymb}
\usepackage{esint}
\usepackage[authoryear]{natbib}

\makeatletter
\numberwithin{equation}{section}
\numberwithin{figure}{section}

\usepackage{amsbsy}
\usepackage{bbm}

\def\H{{\mathcal H}}
\def\m{{\mathfrak m}}
\def\MM{{\mathfrak M}}
\def\S{{\mathcal S}}
\def\IS{{\mathbb S}}
\def\tr{{\rm Tr}}
\def\1{{\mathbbm 1}}
\def\N{{\mathbb{N}}}

\def\R{{\mathbb{R}}}
\def\C{{\mathbb{C}}}
\def\D{{\mathcal D}}
\def\E{{\mathbb{E}}}
\def\EE{{\mathcal{E}}}
\def\P{{\mathbb{P}}}
\def\PP{{\mathcal{P}}}

\def\U{{\mathbb{U}}}

\def\X{{\mathcal{X}}}
\def\XX{{\mathbb{X}}}

\def\HH{\mathbb{H}}

\def\T{{\mathcal T}}

\def\B{{\mathcal B}}
\def\A{{\mathcal A}}

\def\clambda{\check{\lambda}}
\def\cgamma{\check{\gamma}}
\def\Ss{{\mathbb{S}}}

\newcommand{\pa}[1]{\left({#1}\right)}
\newcommand{\cro}[1]{\left[{#1}\right]}
\newcommand{\abs}[1]{\left|{#1}\right|}
\newcommand{\ac}[1]{\left\{{#1}\right\}}
\newcommand{\norm}[1]{\left\|{#1}\right\|}

\newcommand{\beqe}{\begin{eqnarray*}}
\newcommand{\eeqe}{\end{eqnarray*}}
\newcommand{\beq}{\begin{eqnarray}}
\newcommand{\eeq}{\end{eqnarray}}
\newtheorem{thm}{Theorem}
\newtheorem{cor}{Corollary}

\newtheorem{prop}{Proposition}

\newtheorem{defi}{Definition}
\newcommand{\eref}[1]{(\ref{#1})}

\def\<{{\langle}}
\def\>{{\rangle}}

\def\M{{\mathcal M}}
\def\telque{{\ \big |\ }}

\def\sqw{\hbox{\rlap{\leavevmode\raise.3ex\hbox{$\sqcap$}}$%
\sqcup$}}

\usepackage{supertabular}

\makeatother

\usepackage{babel}
\begin{document}

\title{Estimation of the density of a determinantal process}

\author{Yannick Baraud }

\address{Universit\'e de Nice Sophia-Antipolis\\
Laboratoire J.A. Dieudonn\'e\\
UMR CNRS 7351\\
Parc Valrose\\
06108 Nice cedex 02}

\email{baraud@unice.fr}

\date{February 14th, 2013}


\maketitle

\begin{abstract}
We consider the problem of estimating the density $\Pi$ of a determinantal process $N$ from the observation of $n$ independent copies of it. We use an aggregation procedure based on robust testing to build our estimator. We establish non-asymptotic risk bounds with respect to the Hellinger loss and deduce, when $n$ goes to infinity, uniform rates of convergence over classes of densities $\Pi$ of interest. 
\end{abstract}

\section{Introduction}
The starting point of this work goes back to 2007 when Persi Diaconis visited our Laboratory Jean-Alexandre Dieudonn\'e in Nice. At that time, he explained that determinantal processes were emerging in many areas and that there was no statistical procedure to estimate their distributions. Almost five years later, it still seems to be the case. The aim of this paper is therefore to contribute to the study of these processes. Our aim is not only to focus on statistical estimation but also to discuss some related problems. For example, how the class $\D$ of all determinantal densities can be parametrized? Is there an identifiable way of doing it? Another natural question, at least for a Statistician, is to understand how the elements of $\D$ can be approximated. Are there some specific parametric sets that should be used to approximate the densities lying in $\D$? If so, what can be said about the approximation properties of these sets? Finally, given $n$ independent copies of a determinantal process $N$, we propose an estimator of the density $\Pi$ of $N$. We establish non-asymptotic risk bounds for our estimator and deduce uniform rates of convergence over classes of $\Pi$ of interest. It turns out that our estimation strategy is robust with respect to the assumption that $N$ is a determinantal process. This means that the risk bounds we get are not only valid when $\Pi$ belongs to the class $\D$ but also when $\Pi$ is close enough to it (in the Hellinger distance). Our approach is based on $T$-estimation as introduced by Birg\'e~\citeyearpar{MR2219712}. More precisely, we start with a suitable family of models, which typically consists of compact sets of densities, and the role of which is to provide a good approximation of the elements of $\D$. Then, we discretize these models. This results in a family of points $(\Pi_{\m})_{\m\in\MM}$ of $\D$ and we finally use the data in order to select a suitable point among the $\Pi_{\m}$. The way we select this point, which provides our estimator of $\Pi$, is based on robust testing and aims at finding an element among the $\Pi_{\m}$ which is as close as possible to the target density $\Pi$. We establish non-asymptotic risk bounds for our estimator and show how they depend on the approximation properties of the models we started from. Under {\it a posteriori} assumptions on $\Pi$ and for a suitable choice of the models, we specify this bounds and derive rates of convergence. 

For an introduction to determinantal processes, we refer the interested reader to  Lyons~\citeyearpar{MR2031202}, Hough {\it et al}~\citeyearpar{MR2216966} and the books by Anderson {\it et al}~\citeyearpar{MR2760897} and Hough {\it et al}~\citeyearpar{MR2552864}
as well as the references therein. Part of the popularity of determinantal processes comes from the fact that they naturally arise in the study of the eigenvalues of large random matrices. 
Recently, Borodin~{\it et al}~\citeyearpar{MR2721041} showed that these processes are also involved in the process of ``caries'' when adding a column of numbers.

The paper is organized as follows. In Section~\ref{sect-B} we settle the probabilistic background as well as our main notations and conventions. We introduce determinantal processes in Section~\ref{sect-C} and tackle the problem of estimating of their densities in Section~\ref{sect-estimation}. Finally, Section~\ref{sect-P} is devoted to the proofs.

\section{The background}\label{sect-B}
\subsection{Notations and conventions}
Throughout this paper we use the conventions $\sum_{\varnothing}=0$ and $\prod_{\varnothing}=1$ and set $\N^{*}=\N\setminus\{0\}$ and $\R_{+}^{*}=(0,+\infty)$.
Given a finite set $A$, $|A|$ denotes the cardinality of $A$ and for $z\in\C$, $\Re(z)$, $\overline z$ and $|z|$ denote the real part, conjugate and modulus of $z$ respectively.  We denote by $\PP$ the class of all finite subset $J$ of $\N^{*}$ and set $\PP^{*}=\PP\setminus\{\varnothing\}$. Moreover, we set  
\[
\Lambda=\{\lambda\in [0,1]^{\N^{*}},\ \ \abs{\lambda}^{2}=\sum_{j\ge 1}\lambda_{j}^{2}<+\infty\}.
\]
All along, we consider a metric space $(\X,d)$ which we endow with its Borel $\sigma$-field $\B(\X)$ and a $\sigma$-finite measure $\mu$. Roughly speaking, a point process on $(\X,\B(\X))$ will correspond to a random choice of a family of distinct points among $\X$. One should typically think of $\X$ as $\{1,\ldots,p\}$, $\N$, $\R$ or $\R^{p}$ for some positive integer $p$. When $\X$ is not finite, we denote by $\HH$ the Hilbert space of measurable and complex-valued functions $\phi$ on $(\X,\B(\X))$ satisfying
\[
\norm{\phi}^{2}=\int_{\X}\abs{\phi}^{2}d\mu<+\infty.
\]
We endow $\HH$ with the Hermitian inner product defined for $\phi,\psi\in\HH$ by
\[
\<\phi,\psi\>=\int_{\X}\overline \phi \psi d\mu.
\]
For conveniency, we adopt the convention that $\<.,.\>$ is linear with respect to the second argument and not the first one, as usually the case. In order to keep our notation as simple as possible, when $\X$ is finite,  say $\X=\{1,\ldots,p\}$, we embed $\X$ into $\N^{*}$ and use 
\[
\HH=\ell_{2}(\N^{*})=\{\phi\in\C^{\N^{*}},\ \sum_{i\ge 1}\abs{\phi(i)}^{2}<+\infty\}.
\] 
More precisely, a mapping $\phi$ on $\X=\{1,\ldots,p\}$ with values in $\C$ will be viewed as a sequence $(\phi(i))_{i\ge 1}\in\ell_{2}(\N^{*})$ with $\phi(i)=0$ for all $i>p$. Since $\HH$ is an infinite dimensional Hilbert space (whatever $\X$), we may define $\U$ as the set of all orthonormal sequences $\Phi=(\phi_{j})_{j\ge 1}$ in $\HH$. For $J\in \PP^{*}$ and $\Phi\in\U$, we set $\Phi_{J}=(\phi_{j})_{j\in J}$ and given an (ordered) finite subset $\alpha$ of $\X$,  denote by $\Phi_{\alpha,J}$ the $|\alpha|\times |J|$-matrix
\[
\Phi_{\alpha,J}=\pa{\phi_{j}(x)}_{x\in \alpha,j\in J}.
\] 
We extend this notation for rectangle matrices $A$ with entries in $\C$: $A_{\alpha,J}=(A_{i,j})_{i\in \alpha,j\in J}$. Moreover, $A^{*}$ denotes the transpose of the conjugate of $A$, that is, if $A=(A_{i,j})_{i=1,\ldots,k, j=1,\ldots,k'}$, $A^{*}=(\overline A_{j,i})_{j=1,\ldots,k',i=1,\ldots,k}$. 

Finally, we recall that the Hellinger distance $h$ between two densities $p,q$ on a measured space $(E,\EE,\nu)$ is defined by the formula
\[
h^{2}(p,q)={1\over 2}\int_{E}\pa{\sqrt{p}-\sqrt{q}}^{2}d\nu.
\]
For the sake of simplicity, we shall keep the same notation $h$ throughout this paper even though the measured space $(E,\EE,\mu)$ may vary. 

\subsection{The probabilistic background}
In this section, our aim is to introduce the probabilistic background we shall use throughout this paper. We denote by $\XX$ the class of all finite subsets of $\X$ and for $k\in\N$, denote by $\XX_{k}$ the class of those subsets with cardinality $k$. By convention, $\XX_{0}=\{\varnothing\}$. We identify $\XX$ with the set of finite measures of the form
\[
\overline \alpha=\sum_{x\in \alpha}\delta_{x} \ {\rm with}\ \ \alpha\in \XX
\]
and denote the same way $\alpha$ and $\overline \alpha$ so that for all $B\in\B$, 
$\alpha(B)$ means $\abs{\alpha\cap B}$. We equip $\XX$ with the smallest $\sigma$-field $\B(\XX)$ for which the mappings
\begin{eqnarray*}
\XX&\to& \N\\
M_{B}:\alpha &\mapsto& \alpha(B)
\end{eqnarray*}
are measurable for all $B\in\B(\X)$. In particular, the subsets $\XX_{k}=M_{\X}^{-1}(k)$ are measurable for all $k\in \N$. We endow $(\XX,\B(\XX))$ with the measure $L$ defined for all measurable functions $f$ from $\XX$ into $\R_{+}$ by 
\[
\int_{\XX}f(\alpha)dL(\alpha)=f(\varnothing)+\sum_{k\ge 1}{1\over k!}\int_{\X^{\vee k}}f(\{x_{1},\ldots,x_{k}\})d\mu(x_{1})\ldots d\mu(x_{k})
\]
where $\X^{\vee k}$ is the set of all $k$-uplets $(x_{1},\ldots,x_{k})\in\X^{k}$ with distinct coordinates. If $\X$ is finite, say $\X=\{1,\ldots,p\}$, and if $\mu$ is the counting measure on $\X$ then $L$ is merely the counting measure on $\XX$. 

Throughout this paper, a point process $N$ on $(\X,\B(\X))$ is a random variable defined on a probability space $(\Omega,\A,\P)$ with values in $(\XX,\B(\XX),L)$.

\section{Introduction to determinantal processes}\label{sect-C}
In this section, our aim is to define a determinantal process on $(\X,\B(\X))$. To do  so, we adopt the point of view developed in Hough~{\it et al}~\citeyearpar{MR2216966}. In particular, we start with the simpler case of {\it determinantal projection  processes}. 

\subsection{Determinantal projection processes} 
\begin{defi}\label{def-dpp}
Given $J\in\PP^{*}$ and $\Phi\in \U$, a determinantal projection process $N$ of rank $|J|$ with parameter $\Phi_{J}=(\phi_{j})_{j\in J}$ is a point process with density (with respect to $L$) given by
\begin{equation}\label{def-dproj}
\Pi^{\Phi}_{J}(\alpha)=\abs{\det\cro{\Phi_{\alpha,J}}}^{2}\1_{\XX_{|J|}}(\alpha)\ \ {\rm for\ all\ } \alpha\in\XX.
\end{equation}
When $J=\varnothing$, by convention $\Pi^{\Phi}_{\varnothing}=\delta_{\varnothing}$.
\end{defi}

%
%
If the matrix $\Phi_{\alpha,J}=(\phi_{j}(x))_{x\in \alpha,j\in J}$ depends on an ordering on the set $\alpha$, $\abs{\det\cro{\Phi_{\alpha,.}}}$ does not and we shall therefore omit to specify one. It follows from the definition of $\Pi^{\Phi}_{J}$ that with probability 1, $|N(\X)|=|J|$. Hence, a  determinantal projection process $N$ of rank $|J|$ consists of $|J|$ distinct points of $\X$. The location of these points depends on the geometry of the $\phi_{j}$ for $j\in J$. If the $\phi_{j}$ are real-valued, a configuration $\alpha=\{x_{1},\ldots,x_{k}\}$ of points is all the more likely that the volume of the parallelepiped based on the $|J|$ vectors $((\phi_{j}(x_{1}),\ldots,\phi_{j}(x_{k}))_{j\in J}$ is large.

The fact that $\Pi^{\Phi}_{J}$ is a density on $\XX$ might not be clear at first sight. In fact, when $\X=\{1,\ldots,p\}$ this comes the 
the Cauchy-Binet formula: if $A,B$ are $k\times p$ and $p\times k$ matrices respectively with $p\ge k$, the Cauchy-Binet formula asserts that 
\begin{equation}\label{Cauchy-Binet}
\det\cro{AB}=\sum_{\alpha\in\XX_{k}}\det A_{\{1,\ldots,k\},\alpha}\det B_{\alpha,\{1,\ldots,k\}}.\end{equation}
Again, note that this formula is independent of the choice of an ordering on $\alpha$. By using the Cauchy-Binet formula with $B=(\phi_{j}(x))_{x\in\X,j\in J}=\Phi_{\X,J}$, $A=B^{*}$ and by using the fact the family $(\phi_{j})_{j\in J}$ is orthonormal we get
\begin{eqnarray*}
\int_{\XX_{k}}\Pi^{\Phi}_{J}(\alpha)dL(\alpha)&=& \sum_{\alpha\in\XX_{k}}\overline{\det\cro{\Phi_{\alpha,J}}}\det\cro{\Phi_{\alpha,J}}=\sum_{\alpha\in\XX_{k}}\det\cro{\Phi_{J,\alpha}^{*}}\det\cro{\Phi_{\alpha,J}}\\
&=& \det\cro{\Phi_{J,\X}^{*}\Phi_{\X,J}}=\det\cro{\<\phi_{i},\phi_{j}\>_{i,j\in J}}=1.
\end{eqnarray*}
When $\X$ is no longer finite, the Cauchy-Binet formula can be extended by using the identity below from which we can deduce in a similar way as above that $\Pi_{J}^{\Phi}$ is a density.
\begin{prop}\label{CBG}
Let $\Phi=(\phi_{1},\ldots,\phi_{k})$ and $\Psi=(\psi_{1},\ldots,\psi_{k})$ be two elements of $\HH^{k}$. We have that 
\[
\det\cro{\pa{\<\phi_{i},\psi_{j}\>}_{i,j=1,\ldots,k}}=\int_{\XX_{k}}\det\cro{\Phi_{\{1,\ldots,k\},\alpha}^{*}}\det\cro{\Psi_{\alpha,\{1,\ldots,k\}}}dL(\alpha).
\]
\end{prop}
This identity is known for a long time, especially when $\X$ is a compact interval of $\R$ and $\mu$ the Lebesgue measure on $\X$ (see de Bruijn~\citeyearpar{MR0079647}). The general form of this identity can be found in Baik and Rains~\citeyearpar{MR1844203}.

\subsection{The general case} 
As proved in Hough~{\it et al}~\citeyearpar{MR2216966},  the distribution of a (finite) determinantal process $N$ can be viewed as a mixture of densities of some determinantal projection processes. More precisely, a determinantal process can be defined as follows.
\begin{defi}
Let $\Phi\in \U$ and $\lambda\in \Lambda$. A determinantal process $N$ with parameters $(\Phi,\lambda)$ is a point process with density 
\begin{equation}\label{eq:def-det}
\Pi^{\Phi,\lambda}=\sum_{J\in\PP}p_{J}^{\lambda}\Pi^{\Phi}_{J} \ \ {\rm where}\ \ p^{\lambda}_{J}=\prod_{j\in J}\lambda_{j}^{2}\prod_{j\not\in J}(1-\lambda_{j}^{2})\ {\rm for\ all}\ J\in\PP.
\end{equation}
We use the convention $\Pi^{\Phi}_{\varnothing}=\delta_{\varnothing}$.
\end{defi}

Since $\lambda_{j}\in [0,1]$ for all $j\ge 1$ and
\begin{equation}\label{cond-sum}
\sum_{j\ge 1}\lambda_{j}^{2}<+\infty,
\end{equation}
the numbers $p_{J}^{\lambda}$ are nonnegative and well defined (the infinite product $\prod_{j\not\in J}(1-\lambda_{j}^{2})$ converges for all $J\in \PP$). Besides,  
\[
\sum_{J\in\PP}p_{J}^{\lambda}=\prod_{j\ge 1}(\lambda_{j}^{2}+(1-\lambda_{j}^{2}))=1.
\]
Consequently, $\Pi^{\Phi,\lambda}$ is indeed an (at most countable) mixture of densities. Given $J\in\PP^{*}$, it is not difficult to see that for the particular choice $\lambda=\lambda_{J}=(\1_{j\in J})_{j\ge 1}$, the density $\Pi^{\Phi,\lambda}$ is that of a determinantal projection process with parameter $\Phi_{J}$: indeed, for $J'=J$, $p_{J'}^{\lambda}=1$ and for $J'\neq J$, $p_{J'}^{\lambda}=0$.
 
As explained in Hough~{\it et al}~\citeyearpar{MR2216966}, another way of defining a determinantal process is as follows.
First simulate a sequence $(Z_{j})_{j\ge 1}$ of independent Bernoulli random variables with respective parameters $(\lambda_{j}^{2})_{j\ge 1}$. Consider the subset  $\widehat J$ of those indices $j\ge 1$ such that $Z_{j}=1$. Finally choose $N$ according to a determinantal projection process of rank $|\widehat J|$ with parameter $\Phi_{\widehat J}$. With such a description, Condition~\eref{cond-sum} is easy to understand: together with the Borel-Cantelli lemma, it ensures that $\widehat J$ is finite almost surely. It is also clear that the distribution of a determinantal process remains unchanged if we change the labelling of the pairs $((\lambda_{j},\phi_{j}))_{j\ge1}$. That is,  for all bijection $\sigma$ on $\N^{*}$, the parameters $((\phi_{j})_{j\ge 1},(\lambda_{j})_{j\ge 1})$ and $((\phi_{\sigma(j)})_{j\ge 1},(\lambda_{\sigma(j)})_{j\ge 1})$ lead to the same determinantal distribution. In particular, with no loss of generality, we may assume that the sequence $\lambda=(\lambda_{j})_{j\ge 1}$ is non-increasing with respect to $j$. 

In the literature, one usually associates to a determinantal process $N$ a square integrable kernel $K$ on $\X^{2}$ which defines a self-adjoint compact operator on $\HH$ by the formula
\begin{eqnarray}
\HH&\to& \HH\nonumber\\
T_{K}:\phi&\to& \cro{x\mapsto \int_{\X}K(x,y)\phi(y)d\mu(y)}.\label{def-T}
\end{eqnarray}
The sequences $(\lambda_{j}^{2})_{j\ge 1}$ and $\Phi=(\phi_{j})_{j\ge 1}$ mentioned above correspond then to the eigenvalues and associated eigenvectors of $T_{K}$. Conversely, given the sequences $\lambda=(\lambda_{j})_{j\ge 1}$ and $\Phi=(\phi_{j})_{j\ge 1}$ and provided that $\mu(\X)<+\infty$, the kernel $K$ can be obtained by the fomula (Mercer's Theorem)
\begin{equation}\label{def-K}
K(x,y)=\sum_{j\ge 1}\lambda_{j}^{2}\phi_{j}(x)\overline \phi_{j}(y)
\end{equation}
where the series converge absolutely for almost every $(x,y)\in\X^{2}$. When $\X=\{1,\ldots,p\}$, $K$ is merely (any) $p\times p$ Hermitian matrix with eigenvalues in $[0,1]$. Interestingly, the kernel $K$ can be related to the distribution of $N$ by the following formula which holds for all measurable functions $f$ from $\XX$ into $\R_{+}$
\[
\E\cro{\sum_{\alpha\subset N}f(\alpha)}=\int_{\XX}f(\alpha)\det[K_{\alpha,\alpha}]dL(\alpha)\ \ {\rm where}\ \ \ K_{\alpha,\alpha}=\pa{K(x,y)}_{x\in\alpha,y\in \alpha}.
\]
The mapping $\alpha\mapsto \det[K_{\alpha,\alpha}]$ determines the distribution of $N$ and is called the {\it  correlation function}. When $\X=\{1,\ldots,p\}$, this formula simply says  that for all $\alpha\subset \X$
\[
\P\cro{\alpha\subset N}=\det[K_{\alpha,\alpha}].
\]

\subsection{Hellinger distance and determinantal process}
In the previous section, we have seen  that the distribution of a determinantal process can be parametrized by a pair $(\Phi,\lambda)$ in $ \U\times \Lambda$ and that, conversely, any choice of such a pair allows to define a determinantal process.
The aim of this section is to relate the Hellinger distance between the distributions of two determinantal processes associated to two distinct pairs $(\Phi,\lambda)$ and $(\Psi,\gamma)$ to some distance between these pairs. Again, we start with the simpler case of a determinantal projection process.

\subsubsection{Case of a determinantal projection process}
Let $\Phi=(\phi_{j})_{j\ge 1}$ and $\Psi=(\psi_{j})_{j\ge 1}$ be two elements of $\U$ and  $J,J'$ two elements of $\PP$. If $|J|\neq |J'|$, the supports of $\Pi_{J}^{\Phi}$ and $\Pi_{J'}^{\Psi}$ are disjoint (the densities are supported by $\XX_{|J|}$ and $\XX_{|J'|}$ respectively) and hence $h^{2}(\Pi_{J}^{\Phi},\Pi_{J'}^{\Psi})=1$. If $J=J'=\varnothing$, $\Pi_{J}^{\Phi}=\Pi_{J'}^{\Psi}$ and therefore $h^{2}(\Pi_{J}^{\Phi},\Pi_{J'}^{\Psi})=0$. Consequently, the only case we need to consider is the one where $|J|=|J'|\ge 1$. In fact, as already mentioned, we may re-index one of the two sequences, say $\Psi$, in order to have $J=J'$ without changing the distribution $\Pi^{\Psi}$. By doing so, the following result holds.
\begin{prop}\label{propD}
Let $J\in\PP^{*}$ and $\Phi,\Psi\in \U$. We have,
\begin{eqnarray}
h^{2}(\Pi^{\Phi}_{J},\Pi^{\Psi}_{J})&=&1-\int_{\XX_{k}}\abs{\det\cro{\Phi_{\alpha,J}}}\abs{\det\cro{\Psi_{\alpha,J}}}dL(\alpha)\label{def-iso}\\
&\le& 1-\abs{\det\cro{\pa{\<\phi_{i},\psi_{j}\>}_{i,j\in J}}}\nonumber.
\end{eqnarray}
Moreover, 
\[
h^{2}(\Pi^{\Phi}_{J},\Pi^{\Psi}_{J})\le {5\over 2}\sum_{j\in J}\norm{\phi_{j}-\psi_{j}}^{2}.
\]
\end{prop}
Up to constants, the last inequality is sharp. For example, if $J=\{1\}$, $\Pi_{\{1\}}^{\Phi}$ and $\Pi_{\{1\}}^{\Psi}$ correspond to the two densities on $(\X,\B(\X))$ given by 
\begin{equation}\label{ex1}
\Pi_{\{1\}}^{\Phi}(x)=\abs{\phi_{1}(x)}^{2}\ \ {\rm and}\ \ \Pi_{\{1\}}^{\Psi}(x)=\abs{\psi_{1}(x)}^{2}\ \ {\rm for\ all}\ \ x\in\X
\end{equation}
and hence, if $\phi_{1},\psi_{1}$ are two nonnegative real-valued functions on $(\X,\B(\X))$,
\[
h^{2}\pa{\Pi_{\{1\}}^{\Phi},\Pi_{\{1\}}^{\Psi}}={1\over 2}\int_{\X}\pa{\sqrt{\Pi_{\{1\}}^{\Phi}}-\sqrt{\Pi_{\{1\}}^{\Psi}}}^{2}d\mu={1\over 2}\norm{\phi_{1}-\psi_{1}}^{2}.
\]
Clearly, this equality is no longer true when the nonnegativity assumption on $\phi_{1}$ and $\psi_{1}$ is violated. Nevetheless, Proposition~\ref{propD} says that the inequality
remains true (up to a constant). The proof of this proposition is postponed to Section~\ref{P-pD}.

\subsubsection{The general case}
Since $\Pi^{\Phi,\lambda}$ and $\Pi^{\Psi,\gamma}$ are mixtures, the problem of bounding the Hellinger distance between these two densities amounts to understanding how, more generally, the Hellinger distance behaves with respect to mixtures of densities. More precisely, let $p,q$ be two densities on the measured space $(T,\T,m)$ and $(P_{t})_{t\in T}$ and $(Q_{t})_{t\in T}$ two families of densities on a measured space $(E,\A,\nu)$. What can we say about the Hellinger distance between the two mixtures
\[
P=\int_{T} P_{t}p(t)dm(t)\ \ {\rm and}\ \ Q=\int_{T} Q_{t}q(t)dm(t)
\] 
when we known how far $p$ is from $q$ and the $P_{t}$ from the $Q_{t}$? The following result gives an answer.

\begin{prop}\label{melange}
If $m$ and $\nu$ are both $\sigma$-finite, 
\[
h^{2}(P,Q)\le 2h^{2}(p,q)+2\int_{T}h^{2}(P_{t},Q_{t})q(t)dm(t).
\]
\end{prop}
The proof of this result is postponed to Section~\ref{P-m}. 

We may apply Proposition~\ref{melange} with the choices $T=\PP$ ($m$ being the counting measure on $\PP$), $(E,\A,\nu)=(\XX,\B(\XX),L)$, $p=p^{\lambda}$ the density defined on $\PP$ by $p^{\lambda}(J)=p^{\lambda}_{J}$ for all $J\in\PP$, $q=p^{\gamma}$ defined analogously and for $t=J\in \PP$, $P_{t}=\Pi^{\Phi}_{J}$ and $Q_{t}=\Pi^{\Psi}_{J}$. We obtain the following result the proof of which is detailed in Section~\ref{P-m2}.

\begin{prop}\label{H-det}
Let $\Phi,\Psi\in \U$ and $\lambda,\gamma\in \Lambda$ and set 
\[
\check \lambda=\pa{1-\sqrt{1-\lambda_{j}^{2}}}_{j\ge 1}\ \ {\rm and}\ \ \check \gamma=\pa{1-\sqrt{1-\gamma_{j}^{2}}}_{j\ge 1}.
\]
The following inequalities hold
\begin{eqnarray}
h^{2}(p^{\lambda},p^{\gamma})&\le& \abs{\lambda-\gamma}^{2}+\abs{\check{\lambda}-\check{\gamma}}^{2}\label{eq1}\\
\sum_{J\in\PP}p_{J}^{\gamma}h^{2}(\Pi_{J}^{\Phi},\Pi_{J}^{\Psi})&\le& {5\over 2}\sum_{j\ge 1} \gamma_{j}^{2}\norm{\phi_{j}-\psi_{j}}^{2}.\label{eq2}
\end{eqnarray}
In particular, 
\begin{equation}\label{h-det}
h^{2}(\Pi^{\Phi,\lambda},\Pi^{\Psi,\gamma})\le 2\cro{\abs{\lambda-\gamma}^{2}+\abs{\check{\lambda}-\check{\gamma}}^{2}}+5\sum_{j\ge 1}\gamma_{j}^{2}\norm{\phi_{j}-\psi_{j}}^{2}.
\end{equation}  
\end{prop}

\section{Statistical estimation}\label{sect-estimation}
Throughout this section, we consider a point process $N$ on $(\X,\B(\X))$ with density $\Pi$ with respect to $L$. Given $n$ independent copies $N_{1},\ldots,N_{n}$ of $N$, our aim is to estimate $\Pi$. One may naturally think of $N$ as being a determinantal process which means that $\Pi$ belongs to the set $\D$ of all determinantal distributions. Nevertheless, our result is robust with respect to such an assumption in the sense that $\Pi$ may not belong to $\D$. In this case, one may rather consider $\D$ as an approximation set for $\Pi$. Before turning to the estimation of $\Pi$, we shall first discuss some identifiability issues which are of independent interest and may therefore be skipped.

\subsection{Identifiability and  exterior algebra}\label{S-EA}
When $N$ is a determinantal process, we may write $\Pi=\Pi^{\Phi,\lambda}$ for some pair $(\Phi,\lambda)\in \U\times \Lambda$ or, alternatively, define $\Pi$ from some kernel $K$ on $\X^{2}$ as in \eref{def-K}. If these two approaches provide a parametrization of $\D$, none is identifiable. More precisely, two distinct  pairs in $\U\times \Lambda$ or two distinct kernels may parametrize the same determinantal distribution. This lack of identifiability is already true if one restricts to the simpler class of determinantal projection processes. A simple counter-example can be obtained from~\eref{ex1} with $(\X,\B(\X),\mu)=([0,1],\B([0,1],dx)$ by taking $\phi_{1}(x)=e^{ix}$ and $\psi_{1}(x)=e^{2ix}$. In this case, the corresponding kernels $K_{1}(x,y)=e^{i(x-y)}$ and $K_{2}(x,y)=e^{2i(x-y)}$ are distinct but both parametrize the uniform distribution on $[0,1]$. It is also clear from this counter-example that there is no hope to estimate $\phi_{1}$, which is not identifiable either. 

Consequently, a question arises. How can we define a one-to-one parametrization of $\D$? As we shall see, this problem is rather difficult. In fact,  we shall partially answer this question by restricting ourself to the case where $\X=\{1,\ldots,p\}$ and by focusing on the class $\D_{p,k}$ of all determinantal projection distributions of rank $k$ with $k\in\{1,\ldots,p-1\}$ and $p\ge 2$. It follows from Definition~\ref{def-dpp} that for each element $\Pi\in\D_{p,k}$, there exists an orthonormal family $\phi_{1},\ldots,\phi_{k}$ (which is certainly not unique) such that for all $\alpha\in\XX_{k}$
\begin{equation}\label{def-pi}
\Pi(\alpha)=\abs{\det\cro{\Phi_{\alpha,\{1,\ldots,k\}}}}^{2}.
\end{equation}
Let us now consider the exterior algebra $E=\bigwedge^{k}\C^{p}$ consisting of the sums of $k$-blades $\phi_{1}\wedge\ldots\wedge \phi_{k}$ with $\phi_{1},\ldots,\phi_{k}\in \C^{p}$. Since we shall only use the algebraic properties of these objects and more specifically their connections with determinants, we shall not define them and rather refer the interested reader to Mac Lane and Birkhoff ~\citeyearpar{MR941522}  (Chapter XVI, Section 7). Denoting by $e_{1},\ldots,e_{p}$ the canonical basis of $\C^{p}$, this exterior algebra $E$ can be viewed as a $\C$-linear space, a basis of which being given by the $k$-blades of the form 
\[
e_{\alpha}=e_{i_{1}}\wedge\ldots\wedge e_{i_{k}}
\]
where $\alpha=(i_{1},\ldots,i_{k})$ (with $1\le i_{1}<i_{2}\ldots<i_{k}\le p$) varies among $\XX_{k}$. This linear space can be equipped with an Hermitian inner product $[.,.]$ for which the elements $(e_{\alpha})_{\alpha\in\XX_{k}}$ provide an orthonormal family of $E$. Besides, for a $k$-blade $\phi_{1}\wedge\ldots\wedge \phi_{k}$
\begin{equation}\label{def-cro}
[e_{\alpha},\phi_{1}\wedge\ldots\wedge \phi_{k}]=\det\cro{\Phi_{\alpha,\{1,\ldots,k\}}}\ \ {\rm for\ all}\ \ \alpha\in\XX_{k}.
\end{equation}
Let $\S_{E}$ be the unit sphere of $(E,[.,.])$, $G$ the subset of $\S_{E}$ gathering 
the elements of the form $\phi_{1}\wedge\ldots\wedge \phi_{k}$ for $\phi_{1},\ldots,\phi_{k}$ being an orthonormal family of $\C^{p}$ and 
$G_{+}$ be the subset of $\S_{E}$ defined by 
\[
G_{+}=\{g_{+}=\sum_{\alpha\in\XX_{k}}\abs{[e_{\alpha},g]}e_{\alpha}\telque\ g\in G\}.
\]
It follows from~\eref{def-pi} and~\eref{def-cro} that the mapping 
\begin{eqnarray*}
G_{+} &\to& \D_{p,k}\\
g_{+}&\mapsto& \Pi_{g+}: \alpha\mapsto \abs{\cro{e_{\alpha},g}}^{2}=\abs{\cro{e_{\alpha},g_{+}}}^{2}
\end{eqnarray*}
is surjective. It is also clearly one-to-one and provides thus an identifiable parametrization of the elements of $\D_{p,k}$ by those of $G_{+}$. In fact, if $\Delta$ denotes the Hermitian distance on $E$ defined for $g,g'\in E$ by $\Delta^{2}(g,g')=[g-g',g-g']$, $(G_{+},\Delta)$ and $(\D_{p,k},\sqrt{2} h)$ are isometric: by~\eref{def-iso}, for all $g_{+},g'_{+}\in G_{+}$, 
\begin{eqnarray*}
\Delta^{2}(g_{+},g'_{+})&=&\sum_{\alpha\in\XX_{k}}\abs{\cro{g_{+},e_{\alpha}}-\cro{g'_{+},e_{\alpha}}}^{2}=2\cro{1-\sum_{\alpha\in\XX_{k}}\abs{\cro{g,e_{\alpha}}}\abs{\cro{g',e_{\alpha}}}}\\
&=& 2h^{2}(\Pi_{g+},\Pi_{g'+}).
\end{eqnarray*}
The metric dimension (in the sense given in Birg\'e~\citeyearpar{MR2219712}) of a set of densities is usually closely related to the minimax rate of estimation over this set. Roughly speaking, if the metric dimension of the set is $D$, one can expect that the minimax rate be of order $D/n$. The above isometry shows that the metric dimension $D_{p,k}$ of $(\D_{p,k},h)$ is the same as that of the subset $G_{+}\subset E$ for the Hermitian distance. In  particular $D_{p,k}$ is not larger than the dimension of $E$ (in the usual sense, viewed as a linear space on $\R$), that is $D_{p,k}\le 2\binom{p}{k}$. This upper bound is unfortunately very crude and we shall see that the minimax rates can be much faster. We believe that the metric dimension of $G_{+}$ is actually of order $kp$.

\subsection{The main result}
Let us now turn to the statistical part of this paper. As already mentioned, our aim is to estimate the density $\Pi$ of a point process $N$ from the observation of $n$ independent copies of it. Our estimation strategy is based on $T$-estimation. More precisely, we start with an at most countable family $\{\Pi_{\m},\ \m\in\MM\}$ of candidate determinantal densities, the choice of which will be explained below, and we use a test possessing robustness properties in view of selecting the closest element to $\Pi$ among the $\Pi_{\m}$. We shall not detail the statistical procedure here and rather refer the reader to 
Birg\'e~\citeyearpar{MR2219712} (Theorem~9) or Baraud~\citeyearpar{arXiv:0905.1486}. Nevertheless, in order to give an brief account of the estimation strategies described there, let us merely say that they allow to endow $\{\Pi_{\m},\ \m\in\MM\}$ with a (random) binary relation 
$\propto$ by means of a statistical test based on the observations. Given a pair $(\Pi_{\m},\Pi_{\m'})$ of distinct candidate densities, we either have $\Pi_{\m}\propto \Pi_{\m'}$ or $\Pi_{\m'}\propto \Pi_{\m}$, the test being build in such way that the former (respectively the latter) relation is likely to occur when $h(\Pi,\Pi_{\m})$ is small compared to $h(\Pi,\Pi_{\m'})$ and vice-versa. If the relation $\propto$ were a total order on $\{\Pi_{\m},\ \m\in\MM\}$, a natural idea would be to define the estimator of $\Pi$ as the minimal element of $(\{\Pi_{\m},\ \m\in\MM\},\propto)$. Unfortunately, this is not the case since $\propto$ fails to be transitive in general. A nice idea, which is actually due to Birg\'e, is to define the estimator as the element $\widehat \Pi$ of  $\{\Pi_{\m},\ \m\in\MM\}$ minimizing the quantity $\Pi_{\m}\mapsto {\rm crit}(\Pi_{\m})=\sup\{h(\Pi_{\m},\Pi_{\m'}), \Pi_{\m'}\propto \Pi_{\m}\}$. The property of the test ensures that the value of criterion at $\Pi_{\m}$ is likely to be large when $\Pi_{\m}$ is far from $\Pi$ (provided that there exist some elements $\Pi_{\m'}$ which are closer to $\Pi$) while it is likely to be small as soon as $\Pi_{\m}$ lies in a small enough neighborhood of $\Pi$. With such an estimation strategy, we can design an estimator possessing the following property.

\begin{prop}\label{aggreg}
Let $\mathbf{\Pi}=\{\Pi_{\m},\ \m\in\MM\}$ be an at most countable family of densities on $(\XX,\B(\XX),L)$ and $\pi$ a sub-probability on $\MM$, that is
\[
\sum_{\m\in\MM}\pi(\m)\le 1\ \ {\rm and}\ \ \pi(\m)\ge 0\ \ {\rm for\ all}\ \ m\in\MM.
\]
There exist a universal constant $C>0$ and an estimator $\widehat \Pi=\widehat \Pi(\mathbf{\Pi},\pi)$ solely based on $N_{1},\ldots,N_{n}$ such that whatever the density $\Pi$, 
\[
C \E\cro{h^{2}(\Pi,\widehat \Pi)}\le \inf_{\m\in\MM}\cro{h^{2}(\Pi,\Pi_{\m})+{\log(1/\pi(\m))\over n}}.
\]
\end{prop}

\begin{proof}
Proposition~3 page 363 of Baraud~\citeyearpar{arXiv:0905.1486} (Example~1, Density Estimation) ensures that the random measure $n^{-1}\sum_{i=1}^{n}\delta_{N_{i}}$ with intensity $\Pi$ satisfies the assumption of Corollary~5 page 373 of Baraud~\citeyearpar{arXiv:0905.1486}. The result follows by applying this corollary with $s=\Pi$, $\Lambda=\MM$, $s_{\lambda}=\Pi_{\m}$ and $\overline \Delta(\Pi_{\m})=\log(1/\pi(\m))$ for all $\m\in\MM$, the summation condition~(4) of  Baraud~\citeyearpar{arXiv:0905.1486} being satisfied under the assumption that $\pi$ is a sub-probability.
\end{proof}

Before turning to the choice of the family $\{\Pi_{\m},\ \m\in\MM\}$, let us comment on the role of $\pi$ in our result. When $\pi$ is a probability, it can be interpreted as a prior on the family $\{\Pi_{\m},\ \m\in\MM\}$ and gives thus a bayesian flavor to our approach. Intuitively, our procedure tends to advantage densities $\Pi_{\m}$ associated to values of $\pi(\m)$ which are not too small.

We design our family $\{\Pi_{\m},\ \m\in\MM\}$ in view of possessing good approximation properties with respect to the elements of the class $\D$. Inequality~\eref{h-det} tells us that one can approximate a determinantal density $\Pi^{\Phi,\lambda}$ (with respect to the Hellinger distance) by suitably approximating the sequence $\lambda$ and the functions $\phi_{j}$ of $\Phi$ corresponding to those indices $j$ for which $\lambda_{j}$ is large enough. To do so, we introduce compacts subsets of $\Lambda$ and $\HH$ respectively defined as follows. Concerning $\lambda=(\lambda_{j})_{j\ge 1}$, with no loss of generality, we may assume the sequence is non-increasing with respect to $j$ and it is therefore natural to introduce compact sets of the form
\[
\Lambda_{j}=\{\gamma\in\Lambda, \gamma_{j'}=0\ {\rm for\ all}\ j'>j\}
\]  
for different values of $j\ge 1$. This amounts to approximating $\lambda$ by  the truncated sequence keeping the $j$ first entries of $\lambda$, the others being turned to 0. In order to approximate the $\phi_{j}$, we introduce an at most countable family $\H=(H_{m})_{m\in\M}$ of compact subsets of $(\HH,\norm{\ })$. Examples of such compacts sets will be given in Section~\ref{sect-rate} for the purpose of providing rates of convergence.  Given a compact subset $H$ of $(\HH,\norm{\ })$ and some positive number $\eta$, we denote by $H[\eta]$ a maximal $\eta$-separated subset of $H$, that is, any subset $H'\subset H$ of maximal cardinality satisfying the property: for all $\phi,\phi'\in H'$ with $\phi\neq \phi'$, $\norm{\phi-\phi'}>\eta$. The maximality of $H[\eta]$ implies that for all $\phi\in H$, there exists $\phi'\in H[\eta]$ such that $\norm{\phi-\phi'}\le \eta$. This means that $H[\eta]$ is an $\eta$-net for the compact set $H$. By applying Proposition~\ref{aggreg} to a suitable discretization of the compact sets $\Lambda_{j}$ and $H_{m}$, we deduce the result below. Its proof is detailed in Section~\ref{sect-T}.

\begin{thm}\label{main}
Let $\H=(H_{m})_{m\in\M}$ be an at most countable families of compact subsets $H_{m}$ of  $(\HH,\norm{\ })$ and let $\pi$ be a sub-probability on $\M$.
There exists a density estimator $\widehat \Pi$  such  that whatever the density $\Pi$ on $(\XX,\B(\XX),L)$,
\[
C\E\cro{h^{2}(\Pi,\widehat \Pi)}\le \inf_{\Phi\in\U,\lambda\in\Lambda}\cro{h^{2}(\Pi,\Pi^{\Phi,\lambda})+\inf_{j\ge 1}\pa{\sum_{j'=1}^{j}O(\H,\pi,\phi_{j'})+\sum_{j'>j}\lambda_{j'}^{2}}}
\]
where for all $j\ge 1$,
\[
O(\H,\pi,\phi_{j})=\inf_{m\in\M}\cro{\inf_{\psi\in H_{m}}\norm{\phi_{j}-\psi}^{2}+{1\over n}\log\pa{\abs{H_{m}[1/\sqrt{n}]}n\over \pi(m)}}
\]
and $C$ is a positive universal constant.
\end{thm}

Let us now comment on this risk bound.  The term $h^{2}(\Pi,\Pi^{\Phi,\lambda})$ corresponds to the approximation of $\Pi$ by an element of $\D$. It expresses the fact that our estimation procedure is robust with respect to the assumption that $\Pi$ belongs to $\D$. The quantity $\sum_{j'>j}\lambda_{j'}^{2}$ is the bias term that we get for approximating $\lambda$ by the elements of $\Lambda_{j}$. Given $j\ge 1$ and $m\in\M$, $\inf_{\psi\in H_{m}}\norm{\phi_{j}-\psi}^{2}$ corresponds to the best approximation of $\phi_{j}$ by some element of the compact set $H_{m}$. Enlarging $H_{m}$ (for the inclusion) makes this term smaller but may increase the quantity $n^{-1}\log\pa{\abs{H_{m}[1/\sqrt{n}]}n/\pi(m)}$ which measures in some sense the massiveness of $H_{m}$. The quantity $O(\H,\pi,\phi_{j})$ corresponds to the best trade-off that can achieved between these two terms among the family $\H$. It is typically the bound we would get for estimating the function $\phi_{j}$  alone by a model selection procedure among $\H$ (up to possible extra logarithmic factors). The sum $\sum_{j'=1}^{j}O(\H,\pi,\phi_{j'})$ is therefore the risk bound we get for estimating the $j$ first elements $\phi_{1},\ldots,\phi_{j}$ of $\Phi=(\phi_{j'})_{j'\ge 1}$. 

In order to specify these quantities, let us turn to the following typical situation. Let $(S_{m})_{m\in\M}$ be a family of finite-dimensional subspaces of $\HH$ with respective dimension $D_{m}\ge 1$ (viewed as a linear space on $\R$) and for $m\in\M$, let us take $H_{m}=\S\cap S_{m}$ where $\S$ denotes the unit sphere of $\HH$. The following results hold. 

\begin{prop}\label{prop-approx}
For all $n\ge 1$, 
\begin{equation}\label{eq-prop1}
\log \abs{H_{m}[1/\sqrt{n}]}\le D_{m}\log(2\sqrt{n}+1).
\end{equation}
Besides, for all $\phi\in\S$
\begin{equation}\label{eq-prop2}
\inf_{\psi\in H_{m}}\norm{\phi-\psi}\le 4 \inf_{\psi\in S_{m}}\norm{\phi-\psi}.
\end{equation}
\end{prop}
The first inequality gives a control of the maximal size of a  $1/\sqrt{n}$-separated subset of $H_{m}$. The second one shows that $H_{m}$ and $S_{m}$ share similar approximation properties with respect to the elements of $\S$. The proof of the proposition is delayed to Section~\ref{P-A}. With such a result, we deduce from Theorem~\ref{main} the following corollary.

\begin{cor}\label{maincor}
Let $\IS=(S_{m})_{m\in\M}$ be an at most countable family of finite dimensional subspaces of $(\HH,\norm{\ })$ with dimensions $D_{m}\ge 1$ and let $\pi$ be a sub-probability on $\M$. There exists a density estimator $\widehat \Pi=\widehat \Pi(\IS,\pi)$  such  that whatever the density $\Pi$ on $(\XX,\B(\XX),L)$,
\[
C\E\cro{h^{2}(\Pi,\widehat \Pi)}\le \inf_{\Phi\in\U,\lambda\in\Lambda}\cro{h^{2}(\Pi,\Pi^{\Phi,\lambda})+\inf_{j\ge 1}\pa{\sum_{j'=1}^{j}O(\IS,\pi,\phi_{j'})+\sum_{j'>j}\lambda_{j'}^{2}}}
\]
where for all $j\ge 1$,
\[
O(\IS,\pi,\phi_{j})=\inf_{m\in\M}\cro{\inf_{\psi\in S_{m}}\norm{\phi_{j}-\psi}^{2}+{D_{m}\log n+\log(1/\pi(m))\over n}}
\]
and $C$ is a positive universal constant.
\end{cor}

For illustration, let us consider the elementary situation where $\X=\{1,\ldots,p\}$ and assume that $\Pi\in\D_{p,k}$ is a determinantal projection process on $\X$ of rank $k$ as in Section~\ref{S-EA}. In this case, one can choose $\IS=\{S_{0}\}$ where $S_{0}$ is the linear subspace of dimension $p$ of $\HH=\ell_{2}(\N^{*})$ gathering the elements of the form $(u_{1},\ldots,u_{p},0,\ldots)$ with $(u_{1},\ldots,u_{p})\in\C^{p}$ and  $\pi$ the Dirac mass at $0$. 
Whatever $\Pi\in\D_{p,k}$ there exists $\Phi\in\U$ with $\phi_{1},\ldots,\phi_{k}\in S_{0}\cap\S$ such that $\Pi=\Pi^{\Phi,\lambda}$ with $\lambda_{1}=\ldots=\lambda_{k}=1$ and $\lambda_{j}=0$ for $j>k$. Since $O(\{S_{0}\},\pi,\phi_{j'})\le p\log n/n$ for all $j'\in\{1,\ldots,k\}$ and $\sum_{j'>k}\lambda_{j'}^{2}=0$, by applying Corollary~\ref{maincor} we derive the risk bound
\[
C\E\cro{h^{2}(\Pi,\widehat \Pi)}\le {kp\log n\over n}\ \ \ \ {\rm for\ all}\ \ \  \Pi\in\D_{p,k}.
\]
This inequality shows that the minimax rate of estimation over $\D_{p,k}$ is not larger than $kp\log n/n$. 
Since we expect that the metric dimension of $\D_{p,k}$ is of order $kp$, we believe that the logarithmic factor could probably be dropped.

\subsection{Rates of convergence}\label{sect-rate}
In this section, we assume that $\X=[0,1]^k$ for some integer $k\ge 1$. Our aim is to deduce from Corollary~\ref{maincor} some rates of convergence towards $\Pi$ when it is of the form $\Pi^{\Phi,\lambda}$ for some parameter $(\Phi,\lambda)\in\U\times\Lambda$. To do so, we make some {\it a posteriori} smoothness assumptions on $\Phi=(\phi_{j})_{j\ge 1}$. More precisely, we assume that the $\phi_{j}$ are real-valued and belong to classes $\B_{p,p}^{\beta}([0,1]^k)$ of  (possibly) anisotropic real-valued Besov functions indexed by a number $p\in(0,+\infty]$ and a smoothness parameter $\beta=(\beta_{i})_{i=1,\ldots,k}\in
(0,+\infty)^{k}$. When $p=+\infty$, $\B_{\infty,\infty}^{\beta}([0,1]^k)$ is merely the class of anisotropic $\beta$-H\"olderian functions on $[0,1]^{k}$, which means that a function in  $\B_{\infty,\infty}^{\beta}([0,1]^k)$ is $\beta_{i}$-Holderian  on $[0,1]$ when we keep all the coordinates fixed expect the $i$-th. For a more precise definition of these smoothness classes we refer to Hochmuth~\citeyearpar{MR1884234}, at least when $k=2$.  The definition there can easily be generalized to larger values of $k$. Denoting by $|\phi|_{\beta,p,p}$ the Besov semi-norm of a function $\phi$ in $\B_{p,p}^{\beta}([0,1]^{k})$, we set for any $R>0$
\[
\U_{p,p}^{\beta}(R)=\ac{(\phi_{j})_{j\ge 1}\in\U\telque \phi_{j}\in \B_{p,p}^{\beta}([0,1]^{k}),\ |\phi_{j}|_{\beta,p,p}\le R,\ \forall j\ge 1}
\]
and for $j_{0}\in\N^{*}$,
\[
\U_{p,p}^{\beta}(R,j_{0})=\ac{(\phi_{j})_{j\ge 1}\in\U\telque \phi_{j}\in \B_{p,p}^{\beta}([0,1]^{k}),\ |\phi_{j}|_{\beta,p,p}\le R,\ \forall j=1,\ldots,j_{0}}.
\]
In order to approximate the elements of such class, we use the following result of Akakpo~\citeyearpar{Akakpo09}.

\begin{prop}\label{P-Approx2}
Let $p>0$, $k\in \N^{*}$ and $r\in\N$. There exist a collection of linear spaces 
$(S_{m})_{m\in\M_{k,r}}$ with $\M_{k,r}=\bigcup_{D\ge 1}\M_{k,r}(D)$ and a positive number $C_{k,r}$ such that for all 
positive integer $D$, 
\begin{equation}\label{nbm}
\abs{\M_{k,r}(D)}\le e^{C_{r,k}D}, \ \ \sup_{m\in\M_{k,r}(D)}\dim(S_{m})\le C_{r,k}D
\end{equation}
and 
\begin{equation}
\inf_{m\in \M_{k,r}(D)}\inf_{\psi\in S_{m}}\norm{\phi-\psi}\le
C(k,r,p)|\phi|_{\beta,p,p}D^{-\overline{\beta}/k}
\label{Eq-smooth}
\end{equation}
for all $\phi\in {\B}_{p,p}^{\beta}([0,1]^{k})$ and $\beta$ satisfying
\begin{equation}
\sup_{1\le i\le k}\beta_i<r+1\ \ \mbox{and}\ \
\overline{\beta}=\pa{{1\over k}\sum_{i=1}^{k}{1\over \beta_{i}}}^{-1}>k\left[\left(p^{-1}-2^{-1}\right)\vee0\right].
\label{Eq-Bes1}
\end{equation}
\end{prop}

%

Hereafter, $k$ being fixed, we consider the family of linear spaces $\Ss=(S_{m})_{m\in\M}$ indexed with $\M=\bigcup_{r\ge 0}\M_{k,r}$ (omitting thus the dependency with respect to $k$) and endow $\M$ with the sub-probability $\pi$ defined by
\[
\pi(m)=\inf\ac{e^{-(1+C_{r,k})D-r}\telque\ (r,D)\in \N\times\N^{*}, m\in\M_{k,r}(D)}.
\]
By using the first part of~\eref{nbm}, 
\[
\sum_{m\in\M}\pi(m)\le \sum_{r\ge 0}\sum_{D\ge 1}\sum_{m\in\M_{k,r}(D)}e^{-(1+C_{r,k})D-r}\le \sum_{r\ge 0}e^{-r}\sum_{D\ge 1}e^{-D}\le 1
\]
and hence $\pi$ is a sub-probability on $\M$.

By applying our Corollary~\ref{maincor} with the family $(S_{m})_{m\in\M}$ and this sub-probability $\pi$, we deduce the following uniform rates of convergence over the classes of densities of determinantal projection processes of rank $j_{0}\ge 1$ and parameter $\Phi$ belonging to $\U_{p,p}^{\beta}(R,j_{0})$. The proof of the result is delayed to Section~\ref{S-R}.

\begin{prop}\label{rate-proj}
There exists an estimator $\widehat \Pi$ such that for all $j_{0}\ge 1$, $R>0$, $\beta\in (0,+\infty)^{k}$ and $p\in(0,+\infty]$ such that $\overline \beta>k\left[\left(p^{-1}-2^{-1}\right)\vee 0\right]$, we have 
\begin{eqnarray*}
\sup_{\Phi\in \U_{p,p}^{\beta}(R,j_{0})}\!\!\!\!\!\!\!\E\cro{h^{2}(\Pi^{\Phi}_{\{1,\ldots,j_{0}\}},\widehat \Pi)}&\le& C j_{0}\pa{\log n\over n}^{2\overline \beta\over 2\overline \beta+k},
\end{eqnarray*}
where $C$ denotes some positive number depending on $k,R,p$ and $\beta$ only.
\end{prop}
When $j_{0}=1$, $\Pi^{\Phi}_{\{1\}}$ is merely a density on $(\X,\B(\X))$ of the form $\abs{\phi_{1}}^{2}$ for some function $\phi_{1}$ of unit norm belonging to  ${\B}_{p,p}^{\beta}([0,1]^{k})$. Note that $\Pi^{\Phi}_{\{1\}}=\abs{\phi_{1}}^{2}$ also belongs to 
${\B}_{p,p}^{\beta}([0,1]^{k})$ and up to the logarithmic factor, the rate we get is the usual one for estimating a density in ${\B}_{p,p}^{\beta}([0,1]^{k})$.

Let us now establish uniform rates of convergence towards more general classes of determinantal densities. To do so, we also need to make {\it a posteriori} assumptions on $\lambda$. More precisely, we assume that it belongs to classes of the form
\begin{eqnarray*}
\Lambda_{\alpha}^{(a)}(A)&=& \ac{\lambda\in\Lambda\telque \sum_{j'>j}\lambda_{j'}^{2}\le A j^{-\alpha},\ \  \forall j\ge 1}\\
&{\rm or}&\\
\Lambda_{\alpha}^{(g)}(A)&=&\ac{\lambda\in\Lambda\telque \sum_{j'>j}\lambda_{j'}^{2}\le Ae^{-\alpha j},\ \  \forall j\ge 1}
\end{eqnarray*}
for some $A,\alpha>0$. These sets contain sequences of $(\lambda_{j})_{j\ge1}$ which are decreasing polynomially and exponentially fast respectively. We get the following result the proof of which is delayed to Section~\ref{S-R}.

\begin{prop}\label{prop-rate}
There exists an estimator $\widehat \Pi$ such that for all $A,\alpha,R>0$, $\beta\in (0,+\infty)^{k}$ and $p\in(0,+\infty]$ such that $\overline \beta>k\left[\left(p^{-1}-2^{-1}\right)\vee0\right]$, we have 
\begin{eqnarray}
\sup_{(\Phi,\lambda)\in \U_{p,p}^{\beta}(R)\times \Lambda_{\alpha}^{(a)}(A)}\!\!\!\!\!\!\!\E\cro{h^{2}(\Pi^{\Phi,\lambda},\widehat \Pi)}&\le& C\pa{\log n\over n}^{2\alpha\overline \beta\over (2\overline \beta+k)(1+\alpha)}\label{eqa}\\
\sup_{(\Phi,\lambda)\in \U_{p,p}^{\beta}(R)\times \Lambda_{\alpha}^{(g)}(A)}\!\!\!\!\!\!\!\E\cro{h^{2}(\Pi^{\Phi,\lambda},\widehat \Pi)}&\le& C\pa{(\log n)^{2+k/(2\overline \beta)}\over n}^{2\overline \beta\over 2\overline \beta+k}\label{eqg}
\end{eqnarray}
where $C$ denotes some positive number depending on $k,A,R,p,\alpha$ and $\beta$ only.
\end{prop}
%
\section{Proofs}\label{sect-P}

\subsection{Proof of Proposition~\ref{propD}}\label{P-pD}
The first equality is clear since by~\eref{def-dproj} the Hellinger affinity between $\Pi^{\Phi}_{J}$ and $\Pi^{\Psi}_{J}$ equals
\begin{eqnarray*}
\rho(\Pi^{\Phi}_{J},\Pi^{\Psi}_{J})&=&\int_{\XX_{k}}\sqrt{\Pi^{\Phi}_{J}(\alpha)\Pi^{\Psi}_{J}(\alpha)}dL(\alpha)\\
&=&\int_{\XX_{k}}\abs{\det\cro{\Phi_{\alpha,J}}}\abs{\det\cro{\Psi_{\alpha,J}}}dL(\alpha).
\end{eqnarray*}
For the second part we use Proposition~\ref{CBG} and get
\begin{eqnarray*}
1-\int_{\XX_{k}}\abs{\det\cro{\Phi_{\alpha,J}}}\abs{\det\cro{\Psi_{\alpha,J}}}dL(\alpha)&\le & 1-\abs{\int_{\XX_{k}}\det\cro{\Phi_{J,\alpha}^{*}}\det\cro{\Psi_{\alpha,J}}dL(\alpha)}\\
&=& 1-\abs{\det\cro{\pa{\<\phi_{i},\psi_{j}\>}_{i,j\in J}}}.
\end{eqnarray*}

Let us now prove the last inequality and  set $a=\sum_{j\in J}\norm{\phi_{j}-\psi_{j}}^{2}$.
If $a>2/5$ then the result is true since the Hellinger distance is bounded by 1. We may therefore assume that $a\le 2/5$. In the remaining part of the proof we consider the linear space $\M_{J\times J}(\C)$ of $|J|\times |J|$ matrices indexed by $J$ with entries in $\C$. We endow $\M_{J\times J}(\C)$ with the Hilbert-Schmidt norm defined by 
\[
\norm{A}=\pa{\sum_{i\in J}\sum_{j\in J}\abs{A_{i,j}}^{2}}^{1/2}.
\]
It is well-known that this norm is sub-multiplicative in the sense that for all $A,B\in \M_{J\times J}(\C)$, $\norm{AB}\le \norm{A}\norm{B}$ and it also satisfies 
\begin{equation}\label{I-tr}
\abs{\tr(AB)}\le \norm{A}\norm{B}.
\end{equation}
One can decompose the matrix $A=(\<\phi_{i},\psi_{j}\>)_{i,j\in J}$ as  $A=D+B$ where $D$ is diagonal with entries $D_{i,i}=\<\phi_{i},\psi_{i}\>$ and $B=A-D$. Since $\norm{\phi_{i}}=\norm{\psi_{i}}=1$ for all $i$ 
\begin{equation}\label{inter1}
\abs{D_{i,i}}\ge \Re(D_{i,i})=1-{\norm{\phi_{i}-\psi_{i}}^{2}\over 2}\ge 1-{a\over 2}>0\ \ \forall i=1,\ldots,k
\end{equation}
and hence, $D$ is non-singular. We may therefore write
\[
\det A=\det D\det(I+M)\ \ {\rm with}\ \ M=D^{-1}B
\]
and since for all $i$, $\sum_{j\in J}|\<\phi_{i},\psi_{j}\>|^{2}\le \norm{\phi_{i}}^{2}=1$,
\[
\norm{M}^{2}=\sum_{i\in J}{1\over |D_{i,i}|^{2}}\sum_{j\neq i}|\<\phi_{i},\psi_{j}\>|^{2}\le \sum_{i\in J}{1-|D_{i,i}|^{2}\over |D_{i,i}|^{2}}=\Delta_{J}^{2}
\]
and by using~\eref{inter1} with the fact that $u\mapsto u^{-2}(1-u^{2})$ is decreasing on $(0,+\infty)$,
\begin{equation}\label{inter2}
\norm{M}^{2}\le \Delta_{J}^{2}\le {1-a/4\over (1-a/2)^{2}}\sum_{j\in J}\norm{\phi_{j}-\psi_{j}}^{2}\le {9\over 16}<1.
\end{equation}
The matrix $I+M$ is therefore non-singular and we may write $I+M=e^{L}$ for some matrix $L\in\M_{J\times J}(\C)$. In fact, 
\[
L=M-{M^{2}\over 2}+\sum_{p\ge 3}(-1)^{p-1}{M^{p}\over p}
\]
where the series converge normally in $(\M_{J\times J}(\C),\norm{\ })$. Moreover, 
\[
\det(I+M)=e^{\tr(L)}.
\]
Since the mapping $L\mapsto \tr(L)$ is linear and continuous on $(\M_{J\times J}(\C),\norm{\ })$ and since $\tr(M)=0$, 
\[
\tr(L)=-{\tr(M^{2})\over 2}+\sum_{p\ge 3}(-1)^{p-1}{\tr(M^{p})\over p}.
\]
By using~\eref{I-tr} and the sub-multiplicative property of the Hilbert-Schmidt norm, we get
\[
\abs{\tr(L)+{\tr(M^{2})\over 2}}\le \sum_{p\ge 3}{\norm{M}^{p}\over p}\le{ \norm{M}^{3}\over 3(1-\norm{M})}
\]
and thus, by using that $\norm{M}\le \Delta_{J}$,
\begin{eqnarray*}
\Re(\tr(L))&\ge& -\Re\pa{{\tr(M^{2})\over 2}}-{ \norm{M}^{3}\over 3(1-\norm{M})}\\
&\ge& -{\norm{M}^{2}\over 2}\pa{1+{2\norm{M}\over 3(1-\norm{M})}}\\
&\ge& -{\Delta_{J}^{2}\over 2}\pa{1+{2\Delta_{J}\over 3(1-\Delta_{J})}}.
\end{eqnarray*}
This inequality together with the fact that $\log u\ge -(1-u)/u$ for all $u>0$ leads to
\begin{eqnarray*}
\abs{\det A}&=&\abs{\det D}e^{\Re(\tr(L))}\\
&\ge&\exp\cro{{1\over 2}\sum_{i\in J}\log\pa{|D_{i,i}|^{2}}-{\Delta_{J}^{2}\over 2}\pa{1+{2\Delta_{J}\over 3(1-\Delta_{J})}}}\\
&\ge& \exp\cro{-{\Delta_{J}^{2}\over 2}-{\Delta_{J}^{2}\over 2}\pa{1+{2\Delta_{J}\over 3(1-\Delta_{J})}}}\\
&=&\exp\cro{-c(\Delta_{J})\Delta_{J}^{2}}
\end{eqnarray*}
with $c(u)=1+u/[3(1-u)]$ for $u\ge 0$. By using the facts that $a\le 2/5$, $c$ is increasing and~\eref{inter2}, we get
\begin{eqnarray*}
h^{2}(\Pi^{\Phi}_{J},\Pi^{\Psi}_{J})&\le& 1-\abs{\det A}\\
&\le&1-\exp\cro{-c\pa{{a(1-a/4)\over (1-a/2)^{2}}}\Delta_{J}^{2}}\\
&\le& c\pa{a(1-a/4)\over (1-a/2)^{2}} {(1-a/4)\over (1-a/2)^{2}}\sum_{j\in J}\norm{\phi_{j}-\psi_{j}}^{2}\\
&\le& {5\over 2}\sum_{j\in J}\norm{\phi_{j}-\psi_{j}}^{2}.
\end{eqnarray*}


\subsection{Proof of Proposition~\ref{melange}}\label{P-m}
Let us set 
\[
R=\int_{T}P_{t}q(t)dm(t).
\]
Since $h^{2}(P,Q)\le 2h^{2}(P,R)+2h^{2}(R,Q)$, it remains to bound each of those terms from above. The measures $\nu$ and $m$ being $\sigma$-finite, we may apply Fubini-Tonnelli theorem. By using the Cauchy-Schwarz inequality, we bound the   first term as follows.
\begin{eqnarray*}
2h^{2}(P,R)&=& \int_{E}\pa{\sqrt{P}-\sqrt{R}}^{2}d\nu= \int_{E}{(P-R)^{2}\over (\sqrt{P}+\sqrt{R})^{2}}d\nu\\
&\le& \int_{E}{\pa{\int_{T}P_{t}(p(t)-q(t))dm(t)}^{2}\over P+R}d\nu\\
&\le& \int_{E}{\pa{\int_{T}\sqrt{P_{t}}(\sqrt{p(t)}-\sqrt{q(t)})\sqrt{P_{t}}(\sqrt{p(t)}+\sqrt{q(t)})dm(t)}^{2}\over P+R}d\nu\\
&\le& \int_{E}\cro{\int_{T}P_{t}\pa{\sqrt{p(t)}-\sqrt{q(t)}}^{2}dm(t)\times {\int_{T}P_{t}(\sqrt{p(t)}+\sqrt{q(t)})^{2}dm(t)\over P+R}}d\nu\\
&\le& 2\int_{T}\cro{\int_{E}P_{t}(x)\pa{\sqrt{p(t)}-\sqrt{q(t)}}^{2}d\nu(x)}dm(t)= 4h^{2}(p,q).
\end{eqnarray*}
Let us now turn to the second term. By using similar arguments, 
\begin{eqnarray*}
2h^{2}(R,Q)&\le& \int_{E}{\pa{\int_{T}(P_{t}-Q_{t})q(t)dm(t)}^{2}\over R+Q}d\nu\\
&=& \int_{E}{\pa{\int_{T}(\sqrt{P_{t}}-\sqrt{Q_{t}})\sqrt{q(t)}\times (\sqrt{P_{t}}+\sqrt{Q_{t}})\sqrt{q(t)}dm(t)}^{2}\over R+Q}d\nu\\
&\le&  \int_{E}\cro{\int_{T}\pa{\sqrt{P_{t}}-\sqrt{Q_{t}}}^{2}q(t)dm(t)\times {\int_{T}(\sqrt{P_{t}}+\sqrt{Q_{t}})^{2}q(t)dm(t)\over R+Q}}d\nu\\
&\le& 4\int_{T}h^{2}(P_{t},Q_{t})q(t)dm(t).
\end{eqnarray*}
We conclude by adding these two upper bounds.

\subsection{Proof of Proposition~\ref{H-det}}\label{P-m2}
Inequality~\eref{h-det} derives from~\eref{eq1},~\eref{eq2} and Proposition~\ref{melange}. Hence, it remains to prove~\eref{eq1} and~\eref{eq2}.

Let us prove~\eref{eq1}. To do so, we set $a^{-1}=e/[2(e-1)]<1$ and prove the stronger inequality
\begin{eqnarray*}
h^{2}(p^{\lambda},p^{\gamma})&\le& a^{-1}\cro{\abs{\lambda-\gamma}^{2}+\abs{\check{\lambda}-\check{\gamma}}^{2}}.
\end{eqnarray*}
If there exists $j\ge 1$ such that $\abs{\lambda_{j}-\gamma_{j}}^{2}+\abs{\check{\lambda}_{j}-\check{\gamma}_{j}}^{2}>a$ then the result is clear since $h^{2}(p^{\lambda},p^{\gamma})\le 1$. Otherwise,
\begin{eqnarray*}
h^{2}(p^{\lambda},p^{\gamma})&=& 1-\sum_{J\in\PP}\prod_{j\in J}\lambda_{j}\gamma_{j}\prod_{j\not\in J}\clambda_{j}\cgamma_{j}= 1-\prod_{j\ge 1}\pa{\lambda_{j}\gamma_{j}+\clambda_{j}\cgamma_{j}}\\
&=& 1-\prod_{j\ge 1}\cro{1-{1\over 2}\pa{\abs{\lambda_{j}-\gamma_{j}}^{2}+\abs{\clambda_{j}-\cgamma_{j}}^{2}}}\\
&=&1-\exp\cro{\sum_{j\ge 1}\log\pa{1-{1\over 2}\pa{\abs{\lambda_{j}-\gamma_{j}}^{2}+\abs{\clambda_{j}-\cgamma_{j}}^{2}}}}
\end{eqnarray*}
and by using that $\log(1-u)\ge [2a^{-1}\log(1-a/2)]u$ for all $u\in[0,a/2]$, we get
\begin{eqnarray*}
h^{2}(p^{\lambda},p^{\gamma})
&\le& {-\log(1-a/2)\over a}\sum_{j\ge 1}\cro{\abs{\lambda_{j}-\gamma_{j}}^{2}+\abs{\clambda_{j}-\cgamma_{j}}^{2}}\\
&=& {1\over a}\cro{\abs{\lambda-\gamma}^{2}+\abs{\clambda-\cgamma}^{2}}
\end{eqnarray*}
with our choice of $a$.

Let us now prove~\eref{eq2}. By using Proposition~\ref{propD} we have that for all $J\in\PP$, $J\neq \varnothing$
\begin{equation}\label{inter3}
h^{2}(\Pi_{J}^{\Phi},\Pi_{J}^{\Psi})\le {5\over 2}\sum_{j\in J}\norm{\phi_{j}-\psi_{j}}^{2}.
\end{equation}
With the convention $\sum_{\varnothing}=0$, this inequality remains true when $J=\varnothing$ since in this case $\Pi_{J}^{\Phi}=\Pi_{J}^{\Psi}=\delta_{\varnothing}$ and thus $h^{2}(\Pi_{J}^{\Phi},\Pi_{J}^{\Psi})=0$. We may therefore write
\begin{eqnarray*}
\sum_{J\in\PP}p_{J}^{\gamma}h^{2}(\Pi_{J}^{\Phi},\Pi_{J}^{\Psi})&\le & {5\over 2}\sum_{J\in\PP}p_{J}^{\gamma}\sum_{j\in J}\norm{\phi_{j}-\psi_{j}}^{2}\\
&\le& {5\over 2}\sum_{j\ge 1} \norm{\phi_{j}-\psi_{j}}^{2}\sum_{J\in\PP, j\in J}p_{J}^{\gamma}\\
&\le& {5\over 2}\sum_{j\ge 1} \gamma_{j}^{2}\norm{\phi_{j}-\psi_{j}}^{2}\sum_{J\in\PP, j\in J}\prod_{j'\in J, j'\neq j}\gamma_{j'}^{2}\prod_{j'\not\in J}(1-\gamma_{j'}^{2})\\
&=&  {5\over 2}\sum_{j\ge 1} \gamma_{j}^{2}\norm{\phi_{j}-\psi_{j}}^{2}\prod_{j'\ge 1, j'\neq j}(\gamma_{j'}^{2}+(1-\gamma_{j'}^{2}))\\
&=& {5\over 2}\sum_{j\ge 1} \gamma_{j}^{2}\norm{\phi_{j}-\psi_{j}}^{2}
\end{eqnarray*}
as claimed.

\subsection{Proof of Theorem~\ref{main}}\label{sect-T}
The proof is based on Proposition~\ref{aggreg} with suitable choices of $(\Pi_{\m})_{\m\in\MM}$ and sub-probability $\pi'$ on $\MM$. 

Let $j$ be some positive integer. We set  
\[
\Lambda_{j}[1/n]=\{\gamma\in \Lambda|\ \ \gamma_{\ell}\in\{i/n, i=1,\ldots n\}\ {\rm if}\  \ell\le j,\ \gamma_{\ell}=0\ \ {\rm otherwise}\},
\]
and 
\[
\U_{j}=\{\Phi_{\{1,\ldots,j\}}|\ \ \Phi\in\U\}\subset \HH^{j}.
\]
We endow $\HH^{j}$ with the distance $d(\cdot,\cdot)$ defined for $(\phi_{j})_{\ell=1,\ldots,j}$ and $(\psi_{j})_{\ell=1,\ldots,j}$ in $\HH^{j}$ by 
\[
d^{2}((\phi_{j})_{\ell=1,\ldots,j},(\psi_{j})_{\ell=1,\ldots,j})=\sum_{\ell=1}^{j}\norm{\phi_{\ell}-\psi_{\ell}}^{2}.
\]
For all $m_{1},\ldots,m_{j}\in \M$ and $\Phi\in \U$, there exist $\widetilde \phi_{1}\in H_{m_{1}}[1/\sqrt{n}],\ldots,\widetilde \phi_{j}\in H_{m_{j}}[1/\sqrt{n}]$ such that for all $\ell\in\{1,\ldots,j\}$, 
\begin{equation}\label{rev-eq00}
\norm{\phi_{\ell}-\widetilde\phi_{\ell}}\le \inf_{\psi\in H_{m_{\ell}}}\norm{\phi_{\ell}-\psi}+1/\sqrt{n}
\end{equation}
since $H_{m_{\ell}}[1/\sqrt{n}]$ is a $1/\sqrt{n}$-net for $H_{m_{\ell}}$. For such an element $(\widetilde\phi_{j})_{\ell=1,\ldots,j}\in\HH^{j}$, there  exists $\overline \Phi\in\U$ such that 
\[
d((\widetilde\phi_{j})_{\ell=1,\ldots,j},\overline \Phi_{\{1,\ldots,j\}})\le \inf_{\Psi\in \U} d((\widetilde\phi_{j})_{\ell=1,\ldots,j},\Psi_{\{1,\ldots,j\}})+{1\over \sqrt{n}}.
\]
Since $\overline \Phi$ only depends on $\widetilde \phi_{1},\ldots,\widetilde\phi_{j}$, the cardinality of the set $\U(j,m_{1},\ldots,m_{j})$ gathering such $\overline \Phi$ when  $\widetilde \phi_{1},\ldots,\widetilde\phi_{j}$ vary among $H_{m_{1}}[1/\sqrt{n}],\ldots,H_{m_{j}}[1/\sqrt{n}]$ respectively is not larger than $\prod_{\ell=1}^{j}\abs{H_{m_{\ell}}[1/\sqrt{n}]}$. Besides, by using that $\Phi\in\U$
\begin{eqnarray*}
d(\Phi_{\{1,\ldots,j\}},\overline \Phi_{\{1,\ldots,j\}})&\le& d(\Phi_{\{1,\ldots,j\}},(\widetilde\phi_{j})_{\ell=1,\ldots,j})+\inf_{\Psi\in \U} d((\widetilde\phi_{j})_{\ell=1,\ldots,j},\Psi_{\{1,\ldots,j\}})+{1\over \sqrt{n}}\\
&\le& 2d(\Phi_{\{1,\ldots,j\}},(\widetilde\phi_{j})_{\ell=1,\ldots,j})+{1\over \sqrt{n}}
\end{eqnarray*}
and hence, by using~\eref{rev-eq00}
\begin{eqnarray}
d^2(\Phi_{\{1,\ldots,j\}},\overline \Phi_{\{1,\ldots,j\}})&\le& 8\sum_{\ell=1}^{j}\norm{\phi_{j}-\widetilde\phi_{j}}^{2}+ {2\over n}\nonumber\\
&\le& 16\sum_{\ell=1}^{j}\inf_{\psi\in H_{m_{\ell}}}\norm{\phi_{\ell}-\psi}^{2}+{16j+2\over n}.\label{rev-eq0}
\end{eqnarray}

For all $\lambda\in\Lambda$, let us set $\widetilde \lambda=(\lambda_{\ell}\1_{\ell\le j})_{\ell\ge 1}\in\Lambda$. Since for all $\ell\ge 1$, $\lambda_{\ell}\in [0,1]$, we have
\begin{eqnarray}
\abs{\lambda-\widetilde \lambda}^{2}+\abs{\check{\lambda}-\check{\widetilde \lambda}}^{2}&=& \sum_{\ell\ge 1}\pa{\abs{\lambda_{\ell}-\widetilde \lambda_{\ell}}^{2}+\abs{\sqrt{1-\lambda_{\ell}^{2}}-\sqrt{1-\widetilde \lambda_{\ell}^{2}}}^{2}}\nonumber\\
&\le& \sum_{\ell>j}\pa{\lambda_{\ell}^{2}+\abs{1-\sqrt{1-\lambda_{\ell}^{2}}}^{2}}\nonumber\\
&\le&  \sum_{\ell>j}\pa{\lambda_{\ell}^{2}+{\lambda_{\ell}^{4}\over \pa{1+\sqrt{1-\lambda_{\ell}^{2}}}^{2}}}\nonumber\\
&\le& 2\sum_{\ell>j}\lambda_{\ell}^{2}.\label{rev-eq1}
\end{eqnarray}
By construction of $\Lambda_{j}[1/n]$, there exists $\overline \lambda\in \Lambda_{j}[1/n]$ such that $\abs{\widetilde \lambda_{\ell}-\overline \lambda_{\ell}}\le 1/n$ for all $\ell\ge 1$ and by using that for all $u,v\in[0,1]$, $\abs{\sqrt{1-u^{2}}-\sqrt{1-v^{2}}}\le \sqrt{2|u-v|}$, for such an element of $\Lambda_{j}[1/n]$, 
\begin{eqnarray}
\abs{\overline \lambda-\widetilde \lambda}^{2}+\abs{\check{\overline \lambda}-\check{\widetilde \lambda}}^{2}&\le& \sum_{\ell=1}^{j}\pa{\abs{\overline \lambda_{\ell}-\widetilde \lambda_{\ell}}^{2}+\abs{\sqrt{1-\overline \lambda_{\ell}^{2}}-\sqrt{1-\widetilde \lambda_{\ell}^{2}}}^{2}}\nonumber\\
&\le& 3\sum_{\ell=1}^{j} \abs{\overline \lambda_{\ell}-\widetilde \lambda_{\ell}}\le {3j\over n}.\label{rev-eq2}
\end{eqnarray}
By using~\eref{rev-eq1} and~\eref{rev-eq2} with the triangular inequality, we obtain that
\begin{equation}\label{rev-eq3}
\abs{\lambda-\overline \lambda}^{2}+\abs{\check{\lambda}-\check{\overline \lambda}}^{2}\le 4\sum_{\ell>j}\lambda_{\ell}^{2}+{6j\over n}.
\end{equation}

By combining~\eref{h-det} with~\eref{rev-eq0} and~ \eref{rev-eq3}, for all $\Phi\in\U$ and $\lambda\in\Lambda$, there exists $(\overline \Phi,\overline \lambda)\in \U(j,m_{1},\ldots,m_{j})\times \Lambda_{j}[1/n]$ such that 
\begin{eqnarray}
h^{2}(\Pi^{\Phi,\lambda},\Pi^{\overline \Phi,\overline \lambda})&\le& 2\cro{\abs{\lambda-\overline \lambda}^{2}+\abs{\check{\lambda}-\check{\overline \lambda}}^{2}}+5\sum_{j\ge 1}\overline \lambda_{j}^{2}\norm{\phi_{j}-\overline \phi_{j}}^{2}\nonumber\\
&\le& 8\sum_{\ell>j}\lambda_{\ell}^{2}+{12j\over n}+5d^{2}(\Phi_{\{1,\ldots,j\}},\overline \Phi_{\{1,\ldots,j\}})\nonumber\\
&\le& 8\sum_{\ell>j}\lambda_{\ell}^{2}+80\sum_{\ell=1}^{j}\inf_{\psi\in H_{m_{\ell}}}\norm{\phi_{\ell}-\psi}^{2}+{92j+10\over n}.\label{rev-eq000}
\end{eqnarray}

Let us now set 
\[
\MM=\bigcup_{j\ge 1}\bigcup_{m_{1}\in\M,\ldots,m_{j}\in\M}\{j\}\times\bigotimes_{\ell=1}^{j}\{m_{j}\}\times\U(j,m_{1},\ldots,m_{j})\times \Lambda_{j}[1/n]
\]
and for $\m=(j,m_{1},\ldots,m_{j},\Psi,\gamma)\in\MM$, $\Pi_{\m}=\Pi^{\Psi,\gamma}$ and 
\[
\pi'(\m)={1\over (2n)^{j}}\prod_{\ell=1}^{j}{\pi(m_{\ell})\over \abs{H_{m_{\ell}}[1/\sqrt{n}]}}.
\]
Since for all $j\ge 1$ and $m_{1},\ldots,m_{j}\in\M$, $\abs{\U(j,m_{1},\ldots,m_{j})}\le \prod_{\ell=1}^{j}\abs{H_{m_{\ell}}[1/\sqrt{n}]}$ and $\abs{\Lambda_{j}[1/n]}\le n^{j}$, 
\begin{eqnarray*}
\sum_{\m\in\MM}\pi'(\m)&\le& \sum_{j\ge 1}\sum_{m_{1},\ldots,m_{j}\in\M}\sum_{\Psi\in\U(j,m_{1},\ldots,m_{j})}\sum_{\gamma\in\Lambda_{j}[1/n]}{1\over (2n)^{j}}\prod_{\ell=1}^{j}{\pi(m_{\ell})\over \abs{H_{m_{\ell}}[1/\sqrt{n}]}}\\
&\le& \sum_{j\ge 1}{1\over 2^{j}}\pa{\sum_{m\in\M}\pi(m)}^{j}\le 1
\end{eqnarray*}
and hence, $\pi'$ is a sub-probability on $\MM$. By using Proposition~\ref{aggreg} with the family of densities $(\Pi_{\m})_{\m\in\MM}$ and the sub-probability $\pi'$ on $\MM$, we obtain an estimator $\widehat \Pi$ for which $\E\cro{h^{2}(\Pi,\widehat \Pi)}$ is, up to a universal constant $C>0$, not larger than 
\begin{eqnarray*}
h^{2}(\Pi,\Pi_{\m})+{\log(1/\pi'(\m))\over n}\le  2h^{2}(\Pi,\Pi^{\Phi,\lambda})+2h^{2}(\Pi^{\Phi,\lambda},\Pi_{\m})+{\log(1/\pi'(\m))\over n}
\end{eqnarray*}
whatever $(\Phi,\lambda)\in\U\times\Lambda$ and $\m\in\MM$. In particular by using~\eref{rev-eq000}, for any choices of $j\ge 1$ and $m_{1},\ldots,m_{j}\in\M$, there exists some $\m=(j,m_{1},\ldots,m_{j},\overline \Phi,\overline \lambda)\in\MM$ such that
\begin{eqnarray*}
\lefteqn{h^{2}(\Pi^{\Phi,\lambda},\Pi_{\m})+{\log(1/\pi'(\m))\over n}}\\
&=& h^{2}(\Pi^{\Phi,\lambda},\Pi^{\overline \Phi,\overline \lambda}) +{1\over n}\sum_{\ell=1}^{j}\log\pa{2\abs{H_{m_{\ell}}[1/\sqrt{n}]}n\over \pi(m_{\ell})}\\
&\le& 102\cro{\sum_{\ell>j}\lambda_{\ell}^{2}+\sum_{\ell=1}^{j}\pa{\inf_{\psi\in H_{m_{\ell}}}\norm{\phi_{\ell}-\psi}^{2}+{1\over n}\log\pa{2\abs{H_{m_{\ell}}[1/\sqrt{n}]}n\over \pi(m_{\ell})}+{1\over n}}}.
\end{eqnarray*}
Finally, we get the result from the fact that $\Phi,\lambda,j,m_{1},\ldots,m_{j}$ are arbitrary.

\subsection{Proof of Proposition~\ref{prop-approx}}\label{P-A}
For all $m\in\M$, $H_{m}[1/\sqrt{n}]$ is a $1/\sqrt{n}$- separated subset of the unit ball of a finite-dimensional linear space $S_{m}$ on $\R$ of dimension $D_{m}$. Consequently, for all $m\in\M$
\[
\log \abs{H_{m}[1/\sqrt{n}]}\le D_{m}\log(2\sqrt{n}+1).
\]
Let $a<1$. If $\inf_{\psi\in S_{m}}\norm{\phi-\psi}\ge a$, then by the triangular inequality, for all $\psi'$ in $H_{m}$
\[
\norm{\phi-\psi'}\le 2\le {2\over a}\inf_{\psi\in S_{m}}\norm{\phi-\psi}.
\]
Otherwise, there exists $\psi\in S_{m}$ such that $\norm{\phi-\psi}<a$. Hence, $\norm{\psi}\ge \norm{\phi}-\norm{\phi-\psi}\ge 1-a>0$. In particular, $\psi\neq 0$ and we may set $\psi'=\psi/\norm{\psi}\in H_{m}$. Since $\norm{\phi}=1$, we have
\begin{eqnarray*}
\norm{\phi-\psi'}&=&\norm{\phi-{\psi\over \norm{\psi}}}=\norm{\phi-{\phi\over \norm{\psi}}+{\phi\over \norm{\psi}}-{\psi\over \norm{\psi}}}\\
&\le& \abs{{\norm{\psi}-\norm{\phi}\over\norm{\psi}}}+{1\over\norm{\psi}}\norm{\phi-\psi}\\
&\le& {2\norm{\phi-\psi}\over \norm{\psi}}\le {2\norm{\phi-\psi}\over 1-a}.
\end{eqnarray*}
We get the result by choosing $a=1/2$.

\subsection{Proofs of Propositions~\ref{rate-proj} and~\ref{prop-rate}}\label{S-R}
By using the collections of linear spaces $\IS$ and our choice of $\pi$, and by using  some classical optimization with respect to $m\in\M$, we get that for all $j\ge 1$
\[
O(\IS,\pi,\phi_{j})\le C (\log n/n)^{2\overline \beta/(2\overline \beta+k)}
\]
where $C$ is a positive constant depending on $R,k,\beta$ and $p$. Up to the logarithmic factor, this bounds correspond to the usual estimation rate over $\B^{\beta}_{p,p}([0,1]^{k})$. When $\lambda=(\1_{j\le j_{0}})_{j\ge 1}$, $\Pi^{\Phi,\lambda}=\Pi^{\Phi}_{\{1,\ldots,j_{0}\}}$ and Corollary~\ref{maincor} leads to Proposition~\ref{rate-proj}. For all $\lambda\in\Lambda_{\alpha}^{(a)}(A)$, we get from Corollary~\ref{maincor} that
\begin{eqnarray*}
C'\E\cro{h^{2}(\Pi,\widehat \Pi)}&\le& \inf_{j\ge 1}\cro{j(\log n/n)^{2\overline \beta/(2\overline \beta+k)}+Aj^{-\alpha}}.
\end{eqnarray*}
The minimum is achieved for $j$ of order $(n/\log n)^{2\overline \beta/[(2\overline \beta+k)(1+\alpha)]}$, which leads to the rate $(\log n/n)^{2\alpha\overline \beta/[(2\overline \beta+k)(1+\alpha)]}$ as claimed. 
The other rate is obtained by arguing similarly and by choosing $j$ of order 
\[
{2\overline \beta\over \alpha(2\overline \beta+k)}\log n.
\]

\thanks{{\bf Acknowledgments} I would like to thank Michel Merle for the very stimulating discussion we had about the identifiability problem raised in Section~\ref{S-EA}.}


\begin{thebibliography}{}

\end{thebibliography}


\begin{thebibliography}{}

\bibitem[Akakpo, 2009]{Akakpo09}
Akakpo, N. (2009).
\newblock {\em Estimation adaptative par selection de partitions en rectangles
  dyadiques}.
\newblock PhD thesis, University Paris XI.

\bibitem[Anderson et~al., 2010]{MR2760897}
Anderson, G.~W., Guionnet, A., and Zeitouni, O. (2010).
\newblock {\em An introduction to random matrices}, volume 118 of {\em
  Cambridge Studies in Advanced Mathematics}.
\newblock Cambridge University Press, Cambridge.

\bibitem[Baik and Rains, 2001]{MR1844203}
Baik, J. and Rains, E.~M. (2001).
\newblock Algebraic aspects of increasing subsequences.
\newblock {\em Duke Math. J.}, 109(1):1--65.

\bibitem[Baraud, 2011]{arXiv:0905.1486}
Baraud, Y. (2011).
\newblock Estimator selection with respect to {H}ellinger-type risks.
\newblock {\em Probab. Theory Relat. Fields}, 151(1-2):353--401.

\bibitem[Birg{\'e}, 2006]{MR2219712}
Birg{\'e}, L. (2006).
\newblock Model selection via testing: an alternative to (penalized) maximum
  likelihood estimators.
\newblock {\em Ann. Inst. H. Poincar\'e Probab. Statist.}, 42(3):273--325.

\bibitem[Borodin et~al., 2010]{MR2721041}
Borodin, A., Diaconis, P., and Fulman, J. (2010).
\newblock On adding a list of numbers (and other one-dependent determinantal
  processes).
\newblock {\em Bull. Amer. Math. Soc. (N.S.)}, 47(4):639--670.

\bibitem[de~Bruijn, 1955]{MR0079647}
de~Bruijn, N.~G. (1955).
\newblock On some multiple integrals involving determinants.
\newblock {\em J. Indian Math. Soc. (N.S.)}, 19:133--151 (1956).


\bibitem[Hochmuth, 2002]{MR1884234}
Hochmuth, R. (2002).
\newblock Wavelet characterizations for anisotropic {B}esov spaces.
\newblock {\em Appl. Comput. Harmon. Anal.}, 12(2):179--208.

\bibitem[Hough et~al., 2006]{MR2216966}
Hough, J.~B., Krishnapur, M., Peres, Y., and Vir{\'a}g, B. (2006).
\newblock Determinantal processes and independence.
\newblock {\em Probab. Surv.}, 3:206--229 (electronic).

\bibitem[Hough et~al., 2009]{MR2552864}
Hough, J.~B., Krishnapur, M., Peres, Y., and Vir{\'a}g, B. (2009).
\newblock {\em Zeros of {G}aussian analytic functions and determinantal point
  processes}, volume~51 of {\em University Lecture Series}.
\newblock American Mathematical Society, Providence, RI.


\bibitem[Lyons, 2003]{MR2031202}
Lyons, R. (2003).
\newblock Determinantal probability measures.
\newblock {\em Publ. Math. Inst. Hautes \'Etudes Sci.}, (98):167--212.

\bibitem[Mac~Lane and Birkhoff, 1988]{MR941522}
Mac~Lane, S. and Birkhoff, G. (1988).
\newblock {\em Algebra}.
\newblock Chelsea Publishing Co., New York, third edition.

\end{thebibliography}
\bibliographystyle{apalike}

\end{document}